\documentclass[11pt]{article}

\usepackage{amsmath}
\usepackage{amsthm}
\usepackage{amsfonts}

\newcommand{\mmp}{\mathbb{P}}

\newcommand{\od}{\overset{d}{=}}
\newcommand{\dod}{\overset{d}{\to}}
\newcommand{\tp}{\overset{P}{\to}}

\newcommand{\me}{\mathbb{E}}
\newcommand{\mr}{\mathbb{R}}
\newcommand{\mn}{\mathbb{N}}

\newcommand{\lin}{\underset{n\to\infty}{\lim}}

\newcommand{\lit}{\underset{t\to\infty}{\lim}}

\newtheorem{thm}{Theorem}[section]
\newtheorem{lemma}[thm]{Lemma}

\theoremstyle{definition}

\theoremstyle{remark}
\newtheorem{rem}[thm]{Remark}
\begin{document}
\title{Functional limit theorems for renewal shot noise processes with increasing response functions}
%
\author{Alexander Iksanov\footnote{ Faculty of Cybernetics, Taras Shevchenko National University of Kyiv, Kyiv-01601, Ukraine\newline e-mail:
iksan@univ.kiev.ua}\footnote{Supported by a grant awarded by the
President of Ukraine (project $\Phi$47/012)}}
\maketitle
\begin{abstract}
\noindent We consider renewal shot noise processes with response
functions which are eventually nondecreasing and regularly varying
at infinity. We prove weak convergence of renewal shot noise
processes, properly normalized and centered, in the space
$D[0,\infty)$ under the $J_1$ or $M_1$ topology. The limiting
processes are either spectrally nonpositive stable L\'{e}vy
processes, including the Brownian motion, or inverse stable
subordinators (when the response function is slowly varying), or
fractionally integrated stable processes or fractionally
integrated inverse stable subordinators (when the index of regular
variation is positive). The proof exploits fine properties of
renewal processes, distributional properties of stable L\'{e}vy
processes and the continuous mapping theorem.

\end{abstract}

\section{Introduction}

Let $(\xi_k)_{k\in\mn}$ be independent copies of a positive random
variable $\xi$. Denote $$S_0:=0, \ \ S_n:=\xi_1+\ldots+\xi_n, \ \
n\in\mn$$ and $$N(t):=\#\{k\in\mn_0: S_k\leq t\}=\inf\{k\in\mn:
S_k>t\}, \ \ t\in\mr.$$ It is clear that $N(t)=0$ for $t<0$.

Let $D:=D[0,\infty)$ denote the Skorohod space of right-continuous
real-valued functions on $[0,\infty)$ with finite limits from the
left. Elements of $D$ are sometimes called {\it c\`{a}dl\`{a}g
functions}. For a c\`{a}dl\`{a}g function $h$, we define
\begin{equation}\label{shot}
X(t):=\sum_{k\geq 0}h(t-S_k)1_{\{S_k\leq
t\}}=\int_{[0,\,t]}h(t-y){\rm d}N(y), \ \ t\geq 0,
\end{equation}
and call $\big(X(t)\big)_{t\geq 0}$ a renewal shot noise process.
The function $h$ is called an {\it impulse response function} or
just {\it response function}. Note that the so defined $X(t)$ is
a.s. finite, for each $t\geq 0$.

Processes \eqref{shot} and more general shot noise processes have
been used to model a lot of diverse phenomena, see, for instance,
\cite{Hsing} and \cite{Verv} and references therein. More recent
contributions \cite{KlupMik1} and \cite{Sam} have discussed
applications in risk theory and finance, respectively. A
non-exhaustive list of works concerning mathematical aspects of
shot noise processes is given in \cite{AIM}.

Since $h$ is c\`{a}dl\`{a}g, for every $t\geq 0$,
$\big(X(ut)\big)_{u\geq 0}$ is a random element taking values in
$D$. Our aim is to prove the weak convergence of, properly
normalized and centered, $X(ut)$ in $D$ under the $J_1$ or $M_1$
topology. In what follows the symbols
$\overset{J_1}{\Rightarrow}$, $\overset{M_1}{\Rightarrow}$ and
$\Rightarrow$ mean that the convergence takes place under the
$J_1$ topology, under the $M_1$ topology or under either of these,
respectively. The $J_1$ topology is the commonly used topology in
$D$ (see \cite{Bil} and \cite{Whitt2}). We recall that $\lin
x_n=x$ in $D[0,T]$, $T>0$, under the $M_1$ topology if $$\lin \inf
\max\big( \underset{t\in [0,\,1]}{\sup}\,|r_n(t)-r(t)|,
\underset{t\in [0,\,1]}{\sup}\,|u_n(t)-u(t)|\big)=0,$$ where the
infimum is taken over all parametric representations $(u,r)$ of
$x$ and $(u_n,r_n)$ of $x_n$, $n\in\mn$. We refer to p.~80-82 in
\cite{Whitt2} for further details and definitions. The $M_1$
topology which like the $J_1$ topology was introduced in
Skorohod's seminal paper \cite{Skor} is not that common. Its
appearances in the probability literature are comparatively rare.
An incomplete list of works which have effectively used the $M_1$
topology in diverse applied problems includes \cite{Avr},
\cite{Bas}, \cite{Rio}, \cite{WhittPang}, \cite{Berg} and
\cite{Tyran}. Remark 12.3.2 in \cite{Whitt2} gives more
references.

The $J_1$ convergence in $D[0,1]$, as $n\to\infty$, of
$\sum_{k\geq 0}h(t-n^{-1}S_k)1_{\{S_k\leq nt\}}$, properly
normalized and centered, to a Gaussian process can be derived from
more general results obtained in \cite{Igl}. When $\big(N(t)\big)$
is the Poisson process, a functional convergence to a Gaussian
process and an infinite variance stable process was proved in
\cite{KlupMik1} (see also \cite{Hein} and references therein) and
\cite{KlupMik3}, respectively, for shot noise processes which are
more general than ours. We are not aware of any papers which would
prove functional limit theorems for the shot noise processes
$X(ut)$ in the case of a {\it general} renewal process
$\big(N(t)\big)$. In particular, in this wider framework a new
technique is needed intended to replace the characteristic
functions approach available in the Poisson case. To some extent,
this has served as the first motivation for the present research.
Secondly (and more importantly), based on the technique developed
in \cite{GneIks} and \cite{GIM} we expect that a particular case
of Theorem \ref{main2} with $h$ being the distribution function of
a positive random variable will form a basis for obtaining
functional limit theorems for the number of occupied boxes in the
Bernoulli sieve (see \cite{GIM} for the definition and further
details).

While the weak convergence of the renewal shot noise processes
with eventually nonincreasing response functions will be
investigated in a forthcoming paper \cite{IMM}, here we only
consider the renewal shot noise processes with eventually
nondecreasing response functions. Theorem \ref{main2} which is our
main result relies heavily upon known functional limit theorems
for $N(t)$. To shorten the presentation the latter are not given
as a separate statement. Rather they are included in Theorem
\ref{main2} as a particular case with $h(y)=1_{[0,\,\infty)}(y)$.
Note that all bounded eventually nondecreasing $h$ with positive
$\lit h(t)$ satisfy \eqref{0909} below with $\beta=0$ and
$\ell^\ast(x)\equiv \lit h(t)$. Therefore these are covered by the
theorem.
\begin{thm}\label{main2}
Let $h:\mr^+\to\mr$ be a locally bounded, right-continuous and
eventually nondecreasing function, and
\begin{equation}\label{0909}
h(x) \ \sim \ x^\beta \ell^\ast(x), \ \ x \to \infty,
\end{equation}
for some $\beta\in [0,\infty)$ and some $\ell^\ast$ slowly varying
at $\infty$.

\noindent (A1) If $\sigma^2:={\rm Var}\,\xi<\infty$ then
$${X(ut)-\mu^{-1}\int_{[0,\,ut]}h(y){\rm d}y\over h(t)\sqrt{\sigma^2\mu^{-3}t}} \ \overset{J_1}{\Rightarrow} \ \int_{[0,\,u]}(u-y)^\beta {\rm d}
W_2(y), \ \ t\to\infty,$$ where $\mu:=\me \xi<\infty$ and
$\big(W_2(u)\big)_{u\geq 0}$ is a Brownian motion.

\noindent (A2) If $\sigma^2=\infty$ and
$$\int_{[0,\,x]}y^2\mmp\{\xi\in{\rm d}y\} \ \sim \ \ell(x), \ \
x\to\infty,$$ for some $\ell$ slowly varying at $\infty$, then
$${X(ut)-\mu^{-1}\int_{[0,\,ut]}h(y){\rm d}y\over h(t)\mu^{-3/2}c(t)} \ \overset{J_1}{\Rightarrow} \ \int_{[0,\,u]}(u-y)^\beta {\rm d}W_2(y), \ \ t\to\infty,$$
where $c(t)$ is any positive continuous function such that $\lit
{t\ell(c(t))\over c^2(t)}=1$ and $\big(W_2(u)\big)_{u\geq 0}$ is a
Brownian motion.

\noindent (A3) If
\begin{equation}\label{2}
\mmp\{\xi>x\} \ \sim \ x^{-\alpha}\ell(x), \ \ x\to\infty,
\end{equation}
for some $\alpha\in (1,2)$ and some $\ell$ slowly varying at
$\infty$, then
$${X(ut)-\mu^{-1}\int_{[0,\,ut]}h(y){\rm d}y\over h(t)\mu^{-1-1/\alpha}c(t)} \ \overset{M_1}{\Rightarrow} \ \int_{[0,\,u]}(u-y)^\beta {\rm d}W_\alpha(y), \ \ t\to\infty,$$
where $c(t)$ is any positive continuous function such that $\lit
{t\ell(c(t))\over c^\alpha(t)}=1$ and
$\big(W_\alpha(u)\big)_{u\geq 0}$ is an $\alpha$-stable L\'{e}vy
process such that $W_\alpha(1)$ has the characteristic function
\begin{equation}\label{st1}
z\mapsto \exp\big\{-|z|^\alpha
\Gamma(1-\alpha)(\cos(\pi\alpha/2)+i\sin(\pi\alpha/2)\, {\rm
sgn}(z))\big\}, \ z\in\mr.
\end{equation}

\noindent (A4) If condition \eqref{2} holds for some $\alpha\in
(0,1)$ then
$${\mmp\{\xi>t\}\over h(t)}X(ut) \
\overset{J_1}{\Rightarrow} \ \int_{[0,\,u]}(u-y)^\beta {\rm
d}V_\alpha(y), \ \ t\to\infty,$$ where
$\big(V_\alpha(u)\big)_{u\geq 0}$ is an inverse $\alpha$-stable
subordinator defined by
$$V_\alpha(u):=\inf\{s\geq 0: D_\alpha(s)>u\},$$ where $\big(D_\alpha(t)\big)_{t\geq 0}$ is an
$\alpha$-stable subordinator with $-\log \me e^{-sD_\alpha(1)}=
\Gamma(1-\alpha)s^\alpha$, $s\geq 0$.
\end{thm}
\begin{rem}
Theorem \ref{main2} does not cover one case for which we have the
following {\it conjecture}:

\noindent (A5) If condition \eqref{2} holds with $\alpha=1$ then
$${m(t)\over h(t)c(t/m(t))}\bigg(X(ut)-{1\over m(c(t/m(t)))}
\int_{[0,\,ut]}h(y){\rm d}y\bigg) \ \overset{M_1}{\Rightarrow}\
\int_{[0,\,u]}(u-y)^\beta {\rm d}W_1(y),$$ where $c(t)$ is any
positive continuous function such that $\lit {t\ell(c(t))\over
c(t)}=1$, $m(t):=\int_{[0,\,t]}\mmp\{\xi>y\}{\rm d}y$, $t>0$, and
$\big(W_1(u)\big)_{u\geq 0}$ is a $1$-stable L\'{e}vy process such
that $W_1(1)$ has the characteristic function
\begin{equation*}\label{stable}
z\mapsto \exp\big\{-|z|(\pi/2-i\log|z|\,{\rm sgn}(z))\big\}, \
z\in\mr.
\end{equation*}
\end{rem}
The rest of the paper is organized as follows. In Section \ref{ip}
we recall a simplified definition of the stochastic integral in
the case when the integrand is a deterministic function. In
Section \ref{lp} we discuss properties of the limiting processes
appearing in Theorem \ref{main2}. The proof of Theorem \ref{main2}
is given in Section \ref{pr}. In Section \ref{who} we discuss an
extension of Theorem \ref{main2} to response functions $h$
concentrated on the whole line. Finally Appendix collects all the
needed auxiliary information.

\section{Defining a stochastic integral via integration by
parts}\label{ip}

There is a general definition of a stochastic integral with
integrand being a locally bounded predictable process and
integrator being a semimartingale, in particular, a L\'{e}vy
process (see, for instance, Theorem 23.4 in \cite{Kalen}).
However, when the integrand is a deterministic function of bounded
variation there is an equivalent definition which is much simpler.
It turns out that the latter stochastic integral can be defined in
terms of usual Lebesgue-Stieltjes integral and integration by
parts.

Let $f, g\in D[a,b]$, $b>a\geq 0$ and $f$ has bounded variation.
Using Lemma \ref{integr} we define the integral
$\int_{(a,b]}f(b-y){\rm d}g(y)$ by formal integration by parts
$$\int_{(a,b]}f(b-y){\rm d}g(y)=f(0-)g(b)-f((b-a)-)g(a)-\int_{(a,b]}g(y){\rm d}f(b-y).$$
Now if $\big(W(y)\big)_{y\geq 0}=\big(W(y,\omega)\big)_{y\geq 0}$
is a L\'{e}vy process (it has paths in $D$) the definition above
with $g(y):=W(y,\omega)$, for each $\omega$, provides a pathwise
construction of the stochastic integral for all $\omega$:
\begin{equation}\label{parts}
\int_{(a,b]}f(b-y){\rm
d}W(y)=f(0-)W(b)-f((b-a)-)W(a)-\int_{(a,b]}W(y){\rm d}f(b-y).
\end{equation}
From this definition and continuity theorem for characteristic
functions we conclude that
\begin{equation}\label{char}
\log \me \exp\bigg({\rm i}t \int_{(a,\,b]}f(b-y){\rm
d}W(y)\bigg)=\int_{(a,\,b]}\log\me \exp\big({\rm
i}tf(b-y)W(1)\big){\rm d}y, \ \ t\in\mr
\end{equation}
(see Lemma 5.1 in \cite{GneIks} for a similar argument). Let
$f^\ast\in L_2[a,b]$ and $\big(W_2(y)\big)_{y\geq 0}$ be a
Brownian motion. Then
$$-\log \me \exp\bigg({\rm i}t \int_{[a,\,b]}f^\ast(y){\rm
d}W_2(y)\bigg)=2^{-1}t^2\int_{[a,\,b]}(f^\ast)^2(y){\rm d}y, \ \
t\in\mr.$$ Hence the random variable $\int_{[a,b]}f^\ast(y){\rm
d}W_2(y)$ has the same law as
$W_2(1)\sqrt{\int_{[a,b]}(f^\ast)^2(y){\rm d}y}$ which implies the
moment formulae to be used in the sequel:
\begin{equation}\label{mo}
\me \bigg(\int_{[a,b]}f^\ast(y){\rm
d}W_2(y)\bigg)^2=\int_{[a,\,b]}(f^\ast)^2(y){\rm d}y, \ \ \me
\bigg(\int_{[a,b]}f^\ast(y){\rm
d}W_2(y)\bigg)^4=3\bigg(\int_{[a,\,b]}(f^\ast)^2(y){\rm
d}y\bigg)^2.
\end{equation}
Of course, all moments of odd orders equal zero.

\section{Properties of the limit processes in Theorem
\ref{main2}}\label{lp}

Recall that $\big(W_2(u)\big)_{u\geq 0}$ denotes a Brownian motion
and, for $\alpha\in (1,2)$, $\big(W_\alpha(u)\big)_{u\geq 0}$
denotes an $\alpha$-stable L\'{e}vy process such that
$W_\alpha(1)$ has the characteristic function given in
\eqref{st1}.

Let $\beta>0$. The limit processes
$\big(Y_{\alpha,\,\beta}(u)\big)_{u\geq 0}$ defined by
\begin{equation}\label{repr1}
Y_{\alpha,\,\beta}(u):=\int_{[0,\,u]}(u-y)^\beta{\rm
d}W_\alpha(y)=\beta \int_{[0,\,u]}(u-y)^{\beta-1}W_\alpha(y){\rm
d}y
\end{equation}
are called the {\it $\alpha$-stable Riemann-Liouville processes}
or {\it fractionally integrated $\alpha$-stable processes} (see,
for instance, \cite{Aur}).

We now establish some properties of the processes
$\big(Y_{\alpha,\,\beta}(u)\big)$.

\noindent (P1) Their paths are continuous a.s.

This follows from the second equality in \eqref{repr1} and Lemma
\ref{co} (a).

\noindent (P2) They are self-similar with Hurst parameter
$\beta+\alpha^{-1}$, i.e., for every $c>0$
$$\big(Y_{\alpha,\,\beta}(cu)\big)_{u\geq 0}\overset{{\rm f.\,d.}}{=} \big(c^{\beta+\alpha^{-1}}Y_{\alpha,\,\beta}(u)\big)_{u\geq 0},$$ where $\overset{{\rm f.\,d.}}{=}$
denotes the equality of finite-dimensional distributions (see
\cite{Embr} for an accessible introduction to the theory of
self-similar processes).

We only prove this property for two-dimensional distributions. For
any $0<u_1<u_2$ and any $\alpha_1, \alpha_2\in\mr$ we have
\begin{eqnarray*}
\beta^{-1}\big(\alpha_1Y_{\alpha,\,\beta}(cu_1)+\alpha_2Y_{\alpha,\,\beta}(cu_2)\big)&=&
\alpha_1\int_{[0,\,cu_1]}(cu_1-y)^{\beta-1}W_\alpha(y){\rm
d}y\\&+&\alpha_2\int_{[0,\,cu_2]}(cu_2-y)^{\beta-1}W_\alpha(y){\rm
d}y\\&=&c^\beta\int_{[0,\,u_2]}\big(\alpha_1(u_1-y)^{\beta-1}1_{[0,u_1]}(y)\\&+&\alpha_2(u_2-y)^{\beta-1}
\big)W_\alpha(cy){\rm d}y\\ &\od&
c^{\beta+\alpha^{-1}}\int_{[0,\,u_2]}\big(\alpha_1(u_1-y)^{\beta-1}1_{[0,u_1]}(y)\\&+&\alpha_2(u_2-y)^{\beta-1}
\big)W_\alpha(y){\rm d}y\\&=&
\beta^{-1}c^{\beta+\alpha^{-1}}\big(\alpha_1Y_{\alpha,\,\beta}(u_1)+\alpha_2Y_{\alpha,\,\beta}(u_2)\big).
\end{eqnarray*}
where the second equality follows by the change of variable, and
the third is a consequence of the self-similarity with parameter
$\alpha^{-1}$ of $\big(W_\alpha(u)\big)$.

\noindent (P3) For fixed $u>0$, $$Y_{\alpha,\,\beta}(u)\od
\int_{[0,\,u]}y^\beta{\rm d}W_\alpha(y)\od
\bigg({u^{\alpha\beta+1}\over \alpha
\beta+1}\bigg)^{1/\alpha}W_\alpha(1).$$ While the first
distributional equality follows from the fact that, for fixed $u$,
$$\big(W_\alpha(u)-W_\alpha(u-y)\big)_{y\in [0,\,u]}\od
\big(W_\alpha(y)\big)_{y\in [0,\,u]},$$ the second is implied by
the equality
\begin{equation*}
\log \me \exp\big({\rm
i}tY_{\alpha,\,\beta}(u)\big)=\int_{[0,\,u]}\log\me \exp\big({\rm
i}t y^\beta W_\alpha(1)\big){\rm d}y, \ \ t\in\mr.
\end{equation*}
(see \eqref{char}).

\noindent (P4) The increments of $\big(Y_{\alpha,\,\beta}(u)\big)$
are neither independent, nor stationary.

Let $0<v<u$ and $\alpha=2$. Since $\big(W_2(u)\big)$ has
independent increments $X_{2,\,\beta}(v)$ and
$\int_{[v,\,u]}(u-y)^\beta{\rm d}W_2(y)$ are independent. Set
$$r_\beta(u,v):=\bigg(\int_{[0,\,v]}\big((u-y)^\beta-(v-y)^\beta\big)^2{\rm d}y\bigg)^{1/2}.$$ It seems that the integral cannot be evaluated in terms of elementary functions.
Fortunately we only have to check that $r_\beta(u,v)\neq 0$, for
{\it some} $v<u$. Using the inequality $(x-y)^2\geq
2^{-1}x^2-y^2$, $x,y\in\mr$ we conclude that $$r^2_\beta(u,v)\geq
\int_{[0,\,v]}\big(2^{-1}(u-y)^{2\beta}-(v-y)^{2\beta}\big){\rm
d}y= {1\over
2\beta+1}\big(2^{-1}(u^{2\beta+1}-(u-v)^{2\beta+1})-v^{2\beta+1}\big).$$
There is a unique solution $x^\ast$ to the equation
$$x^{2\beta+1}-(x-1)^{2\beta+1}=2.$$ Taking any $x>x^\ast \vee 1$
and any $v>0$ we have $r^2_\beta(xv,v)>0$. In view of the
distributional equality
$$\bigg(\int_{[0,\,v]}(v-y)^\beta{\rm d}W_2(y),\int_{[0,\,xv]}\big((xv-y)^\beta-(v-y)^\beta\big){\rm d}W_2(y)\bigg)\od
\bigg(\bigg({v^{2\beta+1}\over 2\beta+1}\bigg)^{1/2}W_2(1),
r_\beta(xv,v) W_2(1)\bigg),$$ $X_{2,\,\beta}(v)$ and
$\int_{[0,\,xv]}\big((xv-y)^\beta-(v-y)^\beta\big){\rm d}W_2(y)$
are strongly dependent.  Therefore, $X_{2,\,\beta}(v)$ and
$X_{2,\,\beta}(xv)-X_{2,\,\beta}(v)$ are not independent.

\noindent Let $\alpha\in (1,2)$. If the increments were
independent the continuous process
$\big(Y_{\alpha,\,\beta}(u)\big)$ would be Gaussian (see Theorem 5
on p.~189 in \cite{GihSk}) which is not the case.

If the increments were stationary the characteristic function of
$Y_{\alpha,\,\beta}(u)-Y_{\alpha,\,\beta}(v)$ for $0<v<u$ would be
a function of $u-v$. This is however not the case as is seen from
formula
$$\log \me \exp\bigg({\rm i}t\big(Y_{\alpha,\,\beta}(u)-Y_{\alpha,\,\beta}(v)\big)\bigg)=\int_{[0,\,u]}\log\me
\exp\bigg({\rm i}t \big((u-y)^\beta-(v-y)^\beta
1_{[0,\,v]}(y)\big) W_\alpha(1)\bigg){\rm d}y, \ \ t\in\mr.$$

Recall that, for $\alpha\in (0,1)$, $\big(V_\alpha(u)\big)_{u\geq
0}$ denotes an inverse $\alpha$-stable subordinator. Let
$\beta>0$. The limit processes
$\big(Z_{\alpha,\,\beta}(u)\big)_{u\geq 0}$ defined by
\begin{equation*}
Z_{\alpha,\,\beta}(u):=\int_{[0,\,u]}(u-y)^\beta{\rm
d}V_\alpha(y)=\beta \int_{[0,\,u]}(u-y)^{\beta-1}V_\alpha(y){\rm
d}y,
\end{equation*}
where the integral is a pathwise Lebesgue-Stieltjes integral, will
be called the {\it fractionally integrated inverse $\alpha$-stable
subordinators}.

We now establish some properties of these processes.

\noindent (Q1) Their paths are continuous a.s.

Obvious.

\noindent (Q2) They are self-similar with Hurst parameter
$\beta+\alpha$.

This is implied by the self-similarity with index $\alpha$ of
$\big(V_\alpha(u)\big)$.

\noindent (Q3) The law of $Z_{\alpha,\,\beta}(u)$ is uniquely
determined by its moments
\begin{equation}\label{mom}
\me \big(Z_{\alpha,\,\beta}(u)\big)^k=u^{k(\alpha+\beta)}{k!\over
\Gamma^k(1-\alpha)}\prod_{j=1}^k
{\Gamma(\beta+1+(j-1)(\alpha+\beta))\over
\Gamma(j(\alpha+\beta)+1)}, \ \ k\in\mn,
\end{equation}
where $\Gamma(\cdot)$ is the gamma function. In particular,
\begin{equation}\label{127} Z_{\alpha,\, \beta}(1) \od
\int_0^R e^{-cZ_\alpha(t)}{\rm d}t,
\end{equation}
where $R$ is a random variable with the standard exponential law
which is independent of $\big(Z_\alpha(u)\big)_{u\geq 0}$ a
drift-free subordinator with no killing and the L\'{e}vy measure
$$\nu_\alpha({\rm d}t)={e^{-t/\alpha}\over (1-e^{-t/\alpha})^{\alpha+1}}1_{(0,\,\infty)}(t){\rm
d}t,$$ and  $c:=(\alpha+\beta)/\alpha$.

From the results obtained in \cite{Mol} it follows that
$\big(V_\alpha(u)\big)$ is a local time at level $0$ for the
$2(1-\alpha)$-dimensional Bessel process. Therefore, \eqref{mom}
is nothing else but a specialization of formula (4.3) in
\cite{Yor}.

One can check that $$\Phi_\alpha(x):=-\log \me
e^{-xZ_\alpha(1)}={\Gamma(1-\alpha)\Gamma(\alpha x+1)\over
\Gamma(\alpha(x-1)+1)}-1, \ \ x\geq 0.$$ Formula \eqref{mom} with
$u=1$ can be rewritten in an equivalent form
\begin{eqnarray*}
\me Z_{\alpha,\,\beta}^k(1)&=&{k!\over (\Phi_\alpha(c)+1)\ldots
(\Phi_\alpha(ck)+1)}\nonumber\\&=& {k!\over
\prod_{j=1}^k(1-\alpha+j(\alpha+\beta)){\rm B}(1-\alpha,
1+k(\alpha+\beta))}, \ \ k\in\mn,
\end{eqnarray*}
where ${\rm B}(\cdot,\cdot)$ is the beta function, which, by
Theorem 2(i) in \cite{BerYor}, entails distributional equality
\eqref{127}. From the inequality
$$\int_0^R e^{-cZ_\alpha(t)}{\rm d}t\leq R,$$ and the fact that
$\me e^{aR}<\infty$, for $a\in (0,1)$, we conclude that the law of
$Z_{\alpha,\, \beta}(1)$ has some finite exponential moments and
thereby is uniquely determined by its moments.

\noindent (Q4) Their increments are not stationary.

When $\alpha+\beta\neq 1$ this follows from the fact that $\me
Z_{\alpha,\beta}(u)$ is a function of $u^{\alpha+\beta}$ rather
than $u$. The case $\alpha+\beta=1$ follows by continuity.

In \cite{Meer} it was shown that $\big(V_\alpha(u)\big)$ does not
have independent increments. Although we believe it is also the
case for $\big(Z_{\alpha,\,\beta}(u)\big)$, we refrain from
investigating this.

\section{Proof of Theorem \ref{main2}}\label{pr}

{\sc Cases (A1)-(A3)}. The functional limit theorems
$$W^{(t)}(u):={N(ut)-ut\over b(t)} \ \Rightarrow \ W_\alpha(u), \ \
t\to\infty,$$ with case dependent $b(t)$ and $W_\alpha(u)$, can be
found, for instance, in Theorem 1b (i) \cite{Bing}.

For $t>0$ set $$X_t(u):={X(ut)-\int_{[0,\,ut]}h(y){\rm d}y\over
b(t)h(t)}, \ \ u\geq 0.$$ Also recall the notation
$$Y_{\alpha,\,\beta}(u):=\beta \int_{[0,\,u]}W_\alpha(y)(u-y)^{\beta-1}{\rm
d}y=\int_{[0,\,u]}(u-y)^\beta{\rm d}W_\alpha(y), \ \ u\geq 0,$$ if
$\beta>0$, and set $Y_{\alpha,\,0}(u):=W_\alpha(u)$, $u\geq 0$, if
$\beta=0$.

We proceed by showing that, in the subsequent analysis, we can
replace $h$ by a nondecreasing and continuous on $\mr^+$ function
$h^\ast$ with $h^\ast(0)=0$ and such that $h^\ast(t)\sim h(t)$,
$t\to\infty$. To this end, we will use the two step reduction.

Suppose we have already proved that
$$X^\ast_t(u):={\int_{[0,\,ut]} h^\ast(ut-y){\rm d}N(y)-\int_{[0,\,ut]}h^\ast(y){\rm d}y\over b(t)h^\ast(t)} \ \Rightarrow \
Y_{\alpha,\,\beta}(u), \ \ t\to\infty.$$ Now to ensure the
convergence $X_t(u)\Rightarrow Y_{\alpha,\,\beta}(u)$,
$t\to\infty$, it suffices to check that, for any $T>0$,
\begin{equation}\label{aux3}
{\underset{u\in
[0,\,T]}{\sup}\,\bigg|\int_{[0,\,ut]}\big(h(ut-y)-h^\ast(ut-y)\big){\rm
d}N(y)\bigg|\over b(t)h(t)} \ \tp \ 0, \ \ t\to\infty,
\end{equation}
and
\begin{equation}\label{aux1}
{\underset{u\in
[0,\,T]}{\sup}\,\bigg|\int_{[0,\,ut]}\big(h(y)-h^\ast(y)\big){\rm
d}y\bigg|\over b(t)h(t)} \ \to \ 0, \ \ t\to\infty.
\end{equation}

\noindent {\sc Step 1}. We first prove an intuitively clear fact
that the behaviour of $h$ near zero does not influence the
asymptotics of $X_t$. In particular, if, given $a>0$, we replace
$h$ by any c\`{a}dl\`{a}g function $\widehat{h}$ such that
$\widehat{h}(t)=h(t)$ for $t\geq a$ the asymptotics of $X_t$ will
not change. Indeed,
\begin{eqnarray*}
\bigg|\int_{[0,\,u]}\big(h(t(u-y))-\widehat{h}(t(u-y))\big){\rm
d}_yN(ty)\bigg|&=&
\bigg|\int_{(u-a/t,\,u]}\big(h(t(u-y))-\widehat{h}(t(u-y))\big){\rm
d}_yN(ty)\bigg|\\&\leq& \underset{y\in
[0,\,a]}{\sup}\,\big|h(y)-\widehat{h}(y)\big|\big(N(ut)-N(ut-a)\big).
\end{eqnarray*}
Since $h$ and $\widehat{h}$ are c\`{a}dl\`{a}g, they are locally
bounded. After noting that the local boundedness entails the
finiteness of the last supremum, and that in all cases $b$ is
regularly varying with positive index, an appeal to Lemma
\ref{diff} allows us to conclude that, for any $T>0$,
\begin{eqnarray}\label{aux4}
&&{\underset{u\in
[0,\,T]}{\sup}\,\bigg|\int_{[0,\,u]}\big(h(t(u-y))-\widehat{h}(t(u-y))\big){\rm
d}_yN(ty)\bigg|\over b(t)}\nonumber\\&\leq& \underset{y\in
[0,\,a]}{\sup}\,\big|h(y)-\widehat{h}(y)\big|{\underset{u\in
[0,\,T]}{\sup}\,\big(N(ut)-N(ut-a)\big)\over b(t)} \ \tp \ 0, \ \
t\to\infty.
\end{eqnarray}
Arguing in a similar but simpler way we conclude that, for any
$T>0$,
\begin{equation}\label{aux2}
{\underset{u\in
[0,\,T]}{\sup}\,\bigg|\int_{[0,\,ut]}\big(h(y)-\widehat{h}(y)\big){\rm
d}y\bigg|\over b(t)} \ \to \ 0, \ \ t\to\infty.
\end{equation}
This justifies the claim. In particular, choosing $a$ large enough
we can make $\widehat{h}$ nondecreasing on $\mr^+$. Besides that,
we will take $\widehat{h}$ such that $\widehat{h}(t)=0$ for $t\in
[0,b]$ for some $b>0$ to be specified later.

\noindent {\sc Step 2}. Set $h^\ast(t):=\me
\widehat{h}((t-\theta)^+)$, where $\theta$ is a random variable
with the standard exponential distribution. It is clear that
$\widehat{h}(t)\geq h^\ast(t)$, $t\geq 0$. By Lemma \ref{red1},
$h^\ast$ is continuous on $\mr^+$ with $h^\ast(0)=0$ and
$h^\ast(t)\sim \widehat{h}(t)\sim h(t)$, $t\to\infty$.
Furthermore,
$$\int_{[0,\,t]}\big(\widehat{h}(y)-h^\ast(y)\big){\rm d}y \ \sim \ h(t), \ \ t\to\infty,$$
which immediately implies
$${\underset{u\in
[0,\,T]}{\sup}\,\bigg|\int_{[0,\,ut]}\big(\widehat{h}(y)-h^\ast(y)\big){\rm
d}y\bigg|\over
b(t)h(t)}={\int_{[0,\,Tt]}\big(\widehat{h}(y)-h^\ast(y)\big){\rm
d}y\over b(t)h(t)} \ \sim \ {T^\beta\over b(t)} \ \to \ 0, \ \
t\to\infty.$$ In combination with \eqref{aux2} the latter proves
\eqref{aux1}.

Now we intend to apply Lemma \ref{red} with $K_1=\widehat{h}$ and
$K_2=h^\ast$. Since $$\lit {\widehat{h}(t)+h^\ast(t)\over
\int_{[0,\,t]}\big(\widehat{h}(y)-h^\ast(y)\big){\rm d}y}=2$$ and
$$\int_{[0,\,Tt]}\big(\widehat{h}(y)-h^\ast(y)\big){\rm d}y \ \sim \ \ T^\beta
h(t), \ \ t\to\infty,$$ and in all cases $b(t)$ is regularly
varying with positive index, we have
\begin{eqnarray*}
&&{\underset{u\in
[0,\,T]}{\sup}\,\bigg|\int_{[0,\,ut]}\big(\widehat{h}(ut-y)-h^\ast(ut-y)\big){\rm
d}N(y)\bigg|\over b(t)h(t)}\\&=&{\underset{u\in [0,\,T]}{\sup}\,
\int_{[0,\,ut]}\big(\widehat{h}(ut-y)-h^\ast(ut-y)\big){\rm
d}N(y)\over b(t)h(t)}  \ \tp \ 0, \ \ t\to\infty.
\end{eqnarray*}
This together with \eqref{aux4} leads to \eqref{aux3}.

By Potter's bound (Theorem 1.5.6 (iii) in \cite{BGT}) for any
chosen $A>1$, $\delta\in (0,\alpha\beta)$ if $\beta>0$ and
$\delta\in (0,1/2)$ if $\beta=0$ (we take $\alpha=2$ in cases (A1)
and (A2)) there exists $t_0$ such that
$${h^{\ast \alpha}(ty)\over h^{\ast \alpha}(t)}\leq Ay^{\alpha\beta-\delta},$$
whenever $y\leq 1$ and $ty\geq t_0$. Choosing in the definition of
$\widehat{h}$ $b=t_0$, i.e., $\widehat{h}(t)=h^\ast(t)=0$ for
$t\in [0,t_0]$ we can and do assume that
\begin{equation}\label{tech}
{h^{\ast \alpha}(Tty)\over h^{\ast \alpha}(Tt)}\leq
Ay^{\alpha\beta-\delta} \ \ \text{and} \ \ {h^\ast(Tt)\over
h^\ast(t)}\leq A \big(T^{\beta+\delta}\vee
T^{\beta-\delta}\big)=:C(T),
\end{equation}
whenever $T>0$, $y\leq 1$, $Tt\geq t_0$ and $t\geq t_0$. The
second inequality in \eqref{tech} is just Potter's bound.

Setting $$h_t(x):=h^\ast(tx)/h^\ast(t),$$ using the fact that
$N(0)=1$ a.s. and integrating by parts, we have, for $t>0$ and
$u>0$
\begin{eqnarray*}
X^\ast_t(u)&=&\int_{[0,\,u]}h_t(u-y){\rm
d}_yW^{(t)}(y)\nonumber\\&=& {h^\ast(ut)\over
b(t)h^\ast(t)}+\int_{(0,\,u]}h_t(u-y){\rm
d}_yW^{(t)}(y)\nonumber\\&=&\int_{(0,\,u]}W^{(t)}(y){\rm
d}_y\big(-h_t(u-y)\big).
\end{eqnarray*}
It suffices to show that\footnote{Although $W^{(t)}$ and
$W_\alpha$ are not necessarily defined on a common probability
space we can assume that by virtue of Skorohod's representation
theorem.},
\begin{equation}\label{22}
\int_{(0,\,u]}\big(W^{(t)}(y)-W_\alpha(y)\big){\rm
d}_y\big(-h_t(u-y)\big) \ \tp \ 0, \ \ t\to\infty,
\end{equation}
in $D$ under the $J_1$ topology in cases (A1) and (A2) and under
the $M_1$ topology in case (A3), and
\begin{equation}\label{23}
\int_{(0,\,u]} W_\alpha(y){\rm d}_y\big(-h_t(u-y)\big) \
\Rightarrow \ \int_{(0,\,u]} W_\alpha(y){\rm
d}_y\big(-(u-y)^\beta\big)=Y_{\alpha,\,\beta}(u), \ \ t\to\infty.
\end{equation}
The convergence of finite dimensional distributions in \eqref{23}
holds by Lemma \ref{impo} and the continuous mapping theorem (see
the proof for case (A4) for more details). Therefore, as far as
relation \eqref{23} is concerned we only have to prove the
tightness.

\noindent {\sc Cases (A1) and (A2)}. If $\lit x_t=x$ in $D$ under
the $J_1$ topology and $x$ is continuous then, for any $T>0$,
$\lit \underset{u\in [0,\,T]}{\sup}\,|x_t(u)-x(u)|=0$. Hence,
using the monotonicity of $h_t$ we obtain
$$\underset{u\in [0,\,T]}{\sup}\,
\bigg|\int_{(0,\,u]}\big(x_t(y)-x(y)\big){\rm
d}_y\big(-h_t(u-y)\big)\bigg|\leq \underset{u\in
[0,\,T]}{\sup}\,|x_t(u)-x(u)|h_t(T) \ \to \ 0, \ \ t\to\infty.$$
Since $\big(W_2(u)\big)$ is a Brownian motion which has a.s.
continuous paths \eqref{22} follows by the continuous mapping
theorem.

By Lemma \ref{co} (b), for each $t>0$, the process on the
left-hand side of \eqref{23} has a.s. continuous paths. Therefore
we will prove that the convergence in \eqref{23} takes place under
the uniform topology in $C[0,\infty)$ which is more than was
claimed in \eqref{23}. To this end, it suffices to show that the
mentioned convergence holds in $C[0,\,T]$, for any $T>0$. We can
write, for any $T>0$,
\begin{eqnarray}\label{xyx}
\int_{(0,\,Tu]}W_2(y){\rm
d}_y\big(-h_t(Tu-y)\big)&=&{h^\ast(Tt)\over
h^\ast(t)}\int_{(0,\,u]}W_2(Ty){\rm
d}_y\big(-h_{Tt}(u-y)\big)\nonumber\\&\overset{\eqref{parts}}{=}&
{h^\ast(Tt)\over h^\ast(t)}\int_{(0,\,u]}h_{Tt}(u-y){\rm
d}W_2(Ty)\nonumber\\&=:&\widehat{X}_t(u), \ \ u\in [0,1].
\end{eqnarray}
Hence it remains to check the tightness of
$\big(\widehat{X}_t(u)\big)$ in $C[0,1]$. With $u,v\in [0,1]$,
$u>v$
\begin{eqnarray*}
&&\bigg({h^\ast(t)\over
h^\ast(Tt)}\bigg)^4\me\big(\widehat{X}_t(u)-\widehat{X}_t(v)\big)^4=
\me \bigg(\int_{[0,\,v]}\big(h_{Tt}(u-y)-h_{Tt}(v-y)\big){\rm
d}_yW_2(Ty)\\&+& \int_{[v,\,u]}h_{Tt}(u-y){\rm d}_yW_2(Ty)
\bigg)^4=
\me\bigg(\int_{[0,\,v]}\big(h_{Tt}(u-y)-h_{Tt}(v-y)\big){\rm
d}_yW_2(Ty)\bigg)^4\\&+& 6\me
\bigg(\int_{[0,\,v]}\big(h_{Tt}(u-y)-h_{Tt}(v-y)\big){\rm
d}_yW_2(Ty)\bigg)^2\me\bigg(\int_{[v,\,u]}h_{Tt}(u-y){\rm
d}_yW_2(Ty)\bigg)^2\\ &+&\me\bigg(\int_{[v,\,u]}h_{Tt}(u-y){\rm
d}_yW_2(Ty)\bigg)^4=3T^2\bigg(\bigg(\int_{[0,\,v]}\big(h_{Tt}(u-y)-h_{Tt}(v-y)\big)^2{\rm
d}y\bigg)^2\\
&+&2\int_{[0,\,v]}\big(h_{Tt}(u-y)-h_{Tt}(v-y)\big)^2{\rm
d}y\int_{[v,\,u]}h^2_{Tt}(u-y){\rm d}y+
\bigg(\int_{[v,\,u]}h^2_{Tt}(u-y){\rm
d}y\bigg)^2\bigg)\\&=&3T^2\bigg(\int_{[0,\,v]}\big(h_{Tt}(u-y)-h_{Tt}(v-y)\big)^2{\rm
d}y+\int_{[v,\,u]}h^2_{Tt}(u-y){\rm d}y\bigg)^2\\&=&
3T^2\bigg(\int_{[v,\,u]}h_{Tt}^2(y){\rm
d}y-2\int_{[0,\,v]}h_{Tt}(v-y)\big(h_{Tt}(u-y)-h_{Tt}(v-y)\big)\bigg)^2\\&\leq&
3T^2\bigg(\int_{[v,\,u]}h_{Tt}^2(y){\rm d}y\bigg)^2.
\end{eqnarray*}
Here the second equality follows since $\big(W_2(u)\big)$ has
independent increments, and the moments of odd orders of the
integrals involved equal zero. The third equality is a consequence
of \eqref{mo}. The last inequality is explained by the fact that
the functions $h_{Tt}$ are nonnegative and nondecreasing.

Hence when $Tt\geq t_0$ and $T\geq t_0$ we have
\begin{eqnarray*}
\me\big(\widehat{X}_t(u)-\widehat{X}_t(v)\big)^4 &\leq& 3T^2
{h^{\ast 4}(Tt)\over h^{\ast 4}(t)}\bigg(\int_{[v,\,u]}h_{Tt}^
2(y){\rm d}y\bigg)^2\overset{\eqref{tech}}{\leq} 3T^2C^4(T) A
\bigg(\int_{[v,\,u]}y^{2\beta-\delta}{\rm
d}y\bigg)^2\\&=&{3T^2C^4(T)A\over
(2\beta-\delta+1)^2}\big(u^{2\beta-\delta+1}-v^{2\beta-\delta+1}\big)^2.
\end{eqnarray*}
If $u<v$ the same inequality holds. Hence, the required tightness
follows by formula (12.51) and Theorem 12.3 in \cite{Bil}.

\noindent {\sc Proof of \eqref{22} for case (A3)}. We first note
that the functions $h_t$ are absolutely continuous with densities
$$h_t^\prime(y)={t\big(h^\ast(ty)-e^{-ty}\int_{[0,\,ty]}h^\ast(x)e^x{\rm d}x\big)\over h^\ast(t)}.$$
The renewal process $N$ has only unit jumps. Hence, $\lit
J(W^{(t)})=0$ a.s., where $J(\cdot)$ denotes the maximum-jump
functional defined in \eqref{ma}. Since $W^{(t)}
\overset{M_1}{\Rightarrow} W_\alpha$, $t\to\infty$, an appeal to
Lemma \ref{impo1} and the continuous mapping theorem completes the
proof.

Before turning to the proof of \eqref{23} in case (A3) let us
recall the following. The process $\big(W_\alpha(u)\big)$ has no
positive jumps, equivalently, the L\'{e}vy measure of
$W_\alpha(1)$ is concentrated on the negative halfline. Therefore,
it follows from Theorem 25.3 in \cite{Sato} and the fact that the
function $x\to (x\vee 1)^\gamma$, $\gamma>0$ is submultiplicative,
that the power moments of all positive orders of $W_\alpha^+(1)$
are finite\footnote{Moreover, the exponential moments of all
positive orders of $W_\alpha(1)$ are finite.}. Also it is
well-known that
\begin{equation}\label{asy}
\mmp\{W_\alpha(1)<-x\} \ \sim \ {\rm const}\,x^{-\alpha}, \ \
x\to\infty.
\end{equation}
For a formal proof one can use the explicit form of characteristic
function of $W_\alpha(1)$, Theorem 8.1.10 in \cite{BGT} and the
fact that the right tail of the law of $W_\alpha(1)$ is very light
(in particular, it is clearly dominated by the left tail).

\noindent {\sc Proof of \eqref{23} for case (A3)}. We will prove
that the convergence in \eqref{23} takes place in $D$ under the
$M_1$ topology. To this end, it suffices to show that the
mentioned convergence holds in $D[0,\,T]$, for any $T>0$. Define
$\widehat{X}_t(u)$ as in \eqref{xyx} but using $W_\alpha$ instead
of $W_2$. Then the task reduces to proving the tightness of the so
defined $\big(\widehat{X}_t(u)\big)$ in $D[0,1]$. By Theorem 1 in
\cite{Taq}, the required tightness will follow once we have proved
that
\begin{equation}\label{ta}
\mmp\{M\big(\widehat{X}_t(u_1), \widehat{X}_t(u),
\widehat{X}_t(u_2)\big)>\varepsilon\}\leq
L\varepsilon^{-\nu}(u_2-u_1)^{1+\rho}, \ \ 0\leq u_1\leq u\leq
u_2\leq 1,
\end{equation}
for large enough $t$ and some positive constants $L$, $\nu$ and
$\rho$, where for $x_1, x_2, x_2\in \mr$ $M(x_1,x_2,x_3):=0$ if
$x_2\in [x_1\wedge x_3, x_1\vee x_3]$, and $:=|x_2-x_1|\wedge
|x_3-x_2|$, otherwise.

We have
\begin{eqnarray*}
&& \mmp\big\{M\big(\widehat{X}_t(u_1), \widehat{X}_t(u),
\widehat{X}_t(u_2)\big)>\varepsilon\big\}\\&=&
\mmp\big\{\big|\widehat{X}_t(u_1)-\widehat{X}_t(u)\big|>\varepsilon,
\big|\widehat{X}_t(u_2)-\widehat{X}_t(u)\big|>\varepsilon,
\widehat{X}_t(u)<\widehat{X}_t(u_1)\wedge
\widehat{X}_t(u_2)\big\}\\&+&
\mmp\big\{\big|\widehat{X}_t(u_1)-\widehat{X}_t(u)\big|>\varepsilon,
\big|\widehat{X}_t(u_2)-\widehat{X}_t(u)\big|>\varepsilon,
\widehat{X}_t(u)>\widehat{X}_t(u_1)\vee
\widehat{X}_t(u_2)\big\}\\&+& \mmp\big\{\widehat{X}_t(u_1)\wedge
\widehat{X}_t(u_2)>\widehat{X}_t(u)+\varepsilon\big\}+\mmp\big\{\widehat{X}_t(u_1)\vee
\widehat{X}_t(u_2)<\widehat{X}_t(u)-\varepsilon\big\}\\&=&
\mmp\big\{\widehat{X}_t(u)-\widehat{X}_t(u_1)>\varepsilon,
\widehat{X}_t(u_2)-\widehat{X}_t(u)<-\varepsilon
\big\}\\&+&\mmp\big\{\widehat{X}_t(u)-\widehat{X}_t(u_1)<-\varepsilon,
\widehat{X}_t(u_2)-\widehat{X}_t(u)>\varepsilon
\big\}\\&=:&I_t(u_1, u, u_2)+J_t(u_1,u,u_2).
\end{eqnarray*}

Using \eqref{xyx} and formula \eqref{char} with characteristic
function of $W_\alpha(1)$ given by \eqref{st1} we arrive at the
distributional equality
\begin{eqnarray*}
{h^\ast(t)\over
h^\ast(Tt)}\big(\widehat{X}_t(u)-\widehat{X}_t(u_1)\big)&=&\int_{(0,\,u_1]}
\big(h_{Tt}(u-y)-h_{Tt}(u_1-y)\big){\rm
d}W_\alpha(Ty)\\&+&\int_{(u_1,\, u]}h_{Tt}(u-y){\rm
d}W_\alpha(Ty)\\&\od&
T^\alpha\bigg(W_\alpha(1)\bigg(\int_{(0,\,u_1]}
\big(h_{Tt}(u-y)-h_{Tt}(u_1-y)\big)^\alpha{\rm
d}y\bigg)^{1/\alpha}\\&+&W_\alpha^\prime(1)\bigg(\int_{(u_1,\,u]}
h^\alpha_{Tt}(u-y){\rm d}y\bigg)^{1/\alpha}\bigg)\\&\od& T^\alpha
W_\alpha(1)\bigg(\int_{(0,\,u_1]}
\big(h_{Tt}(u-y)-h_{Tt}(u_1-y)\big)^\alpha{\rm d}y\\&+&
\int_{(u_1,\,u]} h^\alpha_{Tt}(u-y){\rm d}y\bigg)^{1/\alpha}\\&=:&
T^\alpha W_\alpha(1)a_t(u_1, u),
\end{eqnarray*}
where $W_\alpha^\prime(1)$ and $W_\alpha(1)$ are i.i.d. Similarly
\begin{eqnarray*}
{h^\ast(t)\over
h^\ast(Tt)}\big(\widehat{X}_t(u_2)-\widehat{X}_t(u)\big)&=&\int_{(0,\,u]}
\big(h_{Tt}(u_2-y)-h_{Tt}(u-y)\big){\rm
d}W_\alpha(Ty)\\&+&\int_{(u,\, u_2]}h_{Tt}(u_2-y){\rm
d}W_\alpha(Ty)\\&\od&
T^\alpha\bigg(W_\alpha(1)\bigg(\int_{(0,\,u]}
\big(h_{Tt}(u_2-y)-h_{Tt}(u-y)\big)^\alpha{\rm
d}y\bigg)^{1/\alpha}\\&+&W_\alpha^\ast(1)\bigg(\int_{(u,\,u_2]}
h^\alpha_{Tt}(u_2-y){\rm d}y\bigg)^{1/\alpha}\bigg)\\&=:& T^\alpha
\big(W_\alpha(1)b_t(u,u_2)+W_\alpha^\ast(1)c_t(u,u_2)\big),
\end{eqnarray*}
where $W_\alpha(1)$ and $W_\alpha^\ast(1)$ are i.i.d. and
$W_\alpha(1)$ is the same as in the previous display.

Using the second inequality in \eqref{tech} and setting
$D(T):=T^\alpha C(T)$ we obtain, for large enough $t$,
\begin{eqnarray*}
I_t(u_1,u,u_2)&=&\mmp\bigg\{T^\alpha {h^\ast(Tt)\over h^\ast(t)}
W_\alpha(1)a_t(u_1,u)>\varepsilon, T^\alpha{h^\ast(Tt)\over
h^\ast(t)}\big(W_\alpha(1)b_t(u,u_2)+W_\alpha^\ast(1)c_t(u,u_2)\big)<-\varepsilon
\bigg\}\\&\leq& \mmp\bigg\{W_\alpha(1)>(a_t(u_1,u)D(T))^{-1}
\varepsilon, W_\alpha(1)b_t(u,u_2)+W_\alpha^\ast(1)c_t(u,u_2)<
-D^{-1}(T)\varepsilon \bigg\}\\&\leq&
\mmp\big\{W_\alpha(1)>(a_t(u_1,u)D(T))^{-1}
\varepsilon\big\}\mmp\bigg\{W_\alpha^\ast(1)c_t(u,u_2)<
-D^{-1}(T)\varepsilon\bigg(1+{b_t(u,u_2)\over
a_t(u_1,u)}\bigg)\bigg\}\\&\leq&
\mmp\big\{W_\alpha(1)>(a_t(u_1,u)D(T))^{-1}
\varepsilon\big\}\mmp\big\{W_\alpha^\ast(1)<-(c_t(u,u_2)D(T))^{-1}\varepsilon
\big\}.
\end{eqnarray*}
In view of \eqref{asy}, there exists a positive constant
$q=q(\varepsilon)$ such that
\begin{equation}\label{ine}
\mmp\{W_\alpha^\ast(1)<-x\}\leq qx^{-\alpha},
\end{equation}
whenever $x\geq \varepsilon
(\alpha\beta-\delta+1)^{1/\alpha}(A^{1/\alpha}D(T))^{-1}$ (the
constants $A$ and $\delta$ were defined in the paragraph that
contains formula \eqref{tech}). Further, for large enough $t$,
$$c^\alpha_t(u,u_2)=\int_{[0,\,u_2-u)}h^\alpha_{Tt}(y){\rm d}y\overset{\eqref{tech}}{\leq}A\int_{[0,\,u_2-u]}y^{\alpha\beta-\delta}{\rm d}y={A\over \alpha\beta-\delta+1}
(u_2-u)^{\alpha\beta-\delta+1}\leq {A\over
\alpha\beta-\delta+1}.$$ In view of this inequality \eqref{ine}
can be applied to estimate
\begin{eqnarray*}
I_t(u_1,u,u_2)&\leq&
\mmp\big\{W_\alpha^\ast(1)<-(c_t(u,u_2)D(T))^{-1}\varepsilon \big\}\\
&\leq& qD^\alpha(T)\varepsilon^{-\alpha}c_t^\alpha(u,u_2)\leq
{AqD^\alpha(T)\over
\alpha\beta-\delta+1}\varepsilon^{-\alpha}(u_2-u_1)^{\alpha\beta-\delta+1}.
\end{eqnarray*}
When $\beta>0$ this crude bound suffices. When $\beta=0$ we need a
more refined estimate for
$\mmp\big\{W_\alpha(1)>(a_t(u_1,u)D(T))^{-1} \varepsilon\big\}$.
To this end, we first work towards estimating $a_t(u_1,u)$. Since
$\alpha>1$ and $1-\delta\in (0,1)$,
$$(x+y)^\alpha\geq x^\alpha+y^\alpha \ \ \text{and} \ \ (x+y)^{1-\delta}\leq
x^{1-\delta}+y^{1-\delta} \ \ \text{for all} \ x,y\geq 0.$$Hence
\begin{eqnarray}
a^\alpha_t(u_1,u)&\leq&
\int_{(0,\,u_1]}\big(h^\alpha_{Tt}(u-y)-h^\alpha_{Tt}(u_1-y)\big){\rm
d}y+\int_{(u_1,\,u]}h^\alpha_{Tt}(u-y){\rm
d}y\nonumber\\&=&\int_{(u_1,\,u]}h^\alpha_{Tt}(y){\rm
d}y\overset{\eqref{tech}}{\leq}A\int_{(u_1,u]}y^{-\delta}{\rm
d}y\leq {A\over
1-\delta}\big(u^{1-\delta}-u_1^{1-\delta}\big)\nonumber\\&\leq&
{A\over 1-\delta}(u-u_1)^{1-\delta}\leq {A\over
1-\delta}(u_2-u_1)^{1-\delta}.\label{au}
\end{eqnarray}
Using \eqref{au} and Markov inequality we conclude that
\begin{eqnarray*}
\mmp\big\{W_\alpha(1)>(a_t(u_1,u)D(T))^{-1}
\varepsilon\big\}&\leq& \me (W_\alpha^+(1))^\alpha
D^\alpha(T)\varepsilon^{-\alpha}a^\alpha_t(u_1,u)\\&\leq& {\me
(W_\alpha^+(1))^\alpha AD^\alpha(T)\over
1-\delta}\varepsilon^{-\alpha}(u_2-u_1)^{1-\delta}.
\end{eqnarray*}
Combining pieces together leads to the inequality
$$I_t(u_1,u,u_2)\leq {\me
(W_\alpha^+(1))^\alpha q A^2D^{2\alpha}(T)\over
(1-\delta)^2}\varepsilon^{-2\alpha}(u_2-u_1)^{2(1-\delta)},$$
which holds for large enough $t$ in the case $\beta=0$ and serves
our needs as $1-\delta>1/2$. Starting with a trivial estimate
$$J_t(u_1,u,u_2)\leq
\mmp\big\{W_\alpha^\ast(1)>(a_t(u_1,u)D(T))^{-1}
\varepsilon\big\}\mmp\big\{W_\alpha(1)<-(c_t(u,u_2)D(T))^{-1}\varepsilon
\big\},$$ we observe that $J_t(u_1,u,u_2)$ is bounded from above
by the same quantity as $I_t(u_1,u,u_2)$. Summarizing we have
proved that \eqref{ta} holds with $L={2\me (W_\alpha^+(1))^\alpha
q A^2D^{2\alpha}(T)\over (1-\delta)^2}$, $\nu=2\alpha$ and
$\rho=1-2\delta $ when $\beta=0$ and with $L={2AqD^\alpha(T)\over
\alpha\beta-\delta+1}$, $\nu=\alpha$ and $\rho=\alpha\beta-\delta$
when $\beta>0$.

\noindent {\sc Case (A4)}. We first note that the functional limit
theorem $$V^{(t)}(u):=\mmp\{\xi>t\}N(ut) \
\overset{J_1}{\Rightarrow} \ V_\alpha(u), \ \ t\to\infty,$$ was
proved in Corollary 3.7 in \cite{Meer}.

We could have proceeded as above, by checking \eqref{22} and
\eqref{23}. However, in the present case the situation is much
easier. Indeed, for each $t>0$, the process
$\big(X^\ast_t(u)\big)$ defined by
$$X^\ast_t(u)=\int_{(0,\,u]}V^{(t)}(y){\rm d}_y\big(-h_t(u-y)\big), \ \ u\geq 0,$$ has nondecreasing paths (as the
convolution of two nondecreasing functions). Recall further that
$\big(V_\alpha(u)\big)$ is a generalized inverse function of a
stable subordinator. Since the paths of the latter are
right-continuous and strictly increasing,
$\big(Z_{\alpha,\,0}(u)\big):=\big(V_\alpha(u)\big)$ has
continuous and nondecreasing sample paths. If $\beta>0$,
$\big(Z_{\alpha,\,\beta}(u)\big)$ has continuous paths by Lemma
\ref{co}. By Theorem 3 in \cite{Bing71} the desired functional
limit theorem will follow once we have established the convergence
of finite dimensional distributions.

We will only investigate the two-dimensional convergence. The
other cases can be treated similarly. Since $X^\ast_t(0)=0$ a.s.,
we only have to prove that, for fixed $0<u<v<\infty$ and any
$\alpha_1, \alpha_2\in\mr$,
\begin{equation}\label{12}
\alpha_1 X^\ast_t(u)+\alpha_2 X^\ast_t(v) \ \dod \
\alpha_1Z_{\alpha,\,\beta}(u)+\alpha_2Z_{\alpha,\,\beta}(v), \ \
t\to\infty.
\end{equation}

For fixed $w>0$ and each $t>0$, define measures $\nu_{t,w}$ and
$\nu_w$ on $[0,w]$ by
$$\nu_{t,w}(c,d]:={h^\ast(t(w-c))-h^\ast(t(w-d))\over h^\ast(t)}, \ \ 0\leq
c<d\leq w$$ and
$$\nu_w(c,d]:=(w-c)^\beta-(w-d)^\beta, \ \ 0\leq c<d\leq w,$$ where
$\beta$ is assumed positive.

\noindent {\sc Case $\beta>0$}. As $t\to\infty$, the measures
$\nu_{t,w}$ weakly converge on $[0,w]$ to $\nu_w$. Now relation
\eqref{12} follows immediately from the equality
\begin{equation}\label{123}
\alpha_1 X^\ast_t(u)+\alpha_2
X^\ast_t(v)=\int_{(0,\,u]}V^{(t)}(y)\big(\alpha_1\nu_{t,u}({\rm
d}y)+\alpha_2 \nu_{t,v}({\rm d}y)\big)+\alpha_2
\int_{(u,\,v]}V^{(t)}(y)\nu_{t,v}({\rm d}y),
\end{equation}
Lemma \ref{impo} and the continuous mapping theorem.

\noindent {\sc Case $\beta=0$}. Now, as $t\to\infty$, the measures
$\nu_{t,w}$ weakly converge on $[0,w]$ to $\delta_w$ (delta
measure). Using \eqref{123} and arguing as before we arrive at
\eqref{12}.

\section{Extension to $h$'s defined on $\mr$}\label{who}

Let $h: \mr\to \mr$ be a right-continuous function with finite
limits from the left. Unlike the situation considered in the
previous sections the corresponding shot noise process is not
necessarily well-defined. However we do not investigate the a.s.
finiteness of $X(t)$ in the most general situation. Rather we
prove that under appropriate assumptions on $h$ which cover most
of practically interesting cases the result of Theorem \ref{main2}
continues to hold.
\begin{thm}\label{main3}
Let $h:\mr \to\mr$ be a right-continuous function with finite
limits from the left such that
$$h(x) \ \sim \ x^\beta \ell^\ast(x), \ \ x \to \infty,$$ for some
$\beta\in [0,\infty)$ and some $\ell^\ast$ slowly varying at
$\infty$. Assume also that $h$ is nondecreasing in the
neighborhood of $+\infty$, and nondecreasing and integrable in the
neighborhood of $-\infty$. Then the result of Theorem \ref{main2}
is valid.
\end{thm}
\begin{proof}
It suffices to prove that, for any $T>0$ and any $c>0$,
\begin{eqnarray*}
{\underset{u\in [0,\,T]}{\sup}\,\sum_{k\geq
0}h(ut-S_k)1_{\{S_k>ut\}}\over t^c}&=& {\underset{u\in
[0,\,Tt]}{\sup}\,\sum_{k\geq 0}h(u-S_k)1_{\{S_k>u\}}\over
t^c}\\&=& {\underset{u\in [0,\,Tt]}{\sup}\,\sum_{k\geq
0}h(u-S_{N(u)+k})\over t^c}  \ \tp \ 0, \ \ t\to\infty.
\end{eqnarray*}
As it was shown at the beginning of the proof of Theorem
\ref{main2}, without loss of generality, we can modify $h$ on any
finite interval in any way that would lead to a right-continuous
resulting function with finite limits from the left. In
particular, we will assume that $\widetilde{h}: [0,\infty)\to
[0,\infty)$ defined by $\widetilde{h}(t)=h(-t)$, $t\geq 0$, is
nonincreasing and $\widetilde{h}(0+)=1$. Note that the
integrability and monotonicity of $h$ in the vicinity of $-\infty$
entail $\lit \widetilde{h}(t)=0$.

The random function $u\to S_{N(u)+k}-u$ attains a.s. its local
minima $S_{k+1+j}-S_j$ at points $S_j$, $j\in\mn$. Hence
$$\underset{u\in [0,\,Tt]}{\sup}\,\sum_{k\geq
0}\widetilde{h}(S_{N(u)+k}-u)=\underset{1\leq j\leq
N(Tt)-1}{\sup}\,\sum_{k\geq
0}\widetilde{h}(S_{k+1+j}-S_j)=:\underset{1\leq j\leq
N(Tt)-1}{\sup}\,\tau_j.$$ Note that the sequence
$(\tau_j)_{j\in\mn_0}$ is stationary. Since
$\int_{[0,\,\infty)}\widetilde{h}(y){\rm d}y<\infty$ implies
$\int_{[0,\,\infty)}\widetilde{h}^p(y){\rm d}y<\infty$ for any
$p>1$, we conclude that $\me \tau_0^p<\infty$ for any $p>0$, by
Theorem 3.7 in \cite{AIM}.

By the weak law of large numbers, for any $\delta>0$, $\lit
\mmp\{N(Tt)>Tt(\mu^{-1}+\delta)\}=0$, where $\mu^{-1}$ is
interpreted as $0$ when $\mu=\infty$. Choose $p>0$ such that
$pc>1$. Then, for any $\varepsilon>0$, $$\mmp\big\{\underset{1\leq
j\leq [Tt(\mu^{-1}+\delta)]}{\sup}\,\tau_j
>\varepsilon t^c\big\}\leq \sum_{j=1}^{[Tt(\mu^{-1}+\delta)]}\mmp\{\tau_0>\varepsilon
t^c\}\leq [Tt(\mu^{-1}+\delta)]\varepsilon^{-p} t^{-pc}\me
\tau_0^p \ \to \ 0, \ \ t\to\infty,$$ by Markov inequality.
Therefore, as $t\to\infty$,
\begin{eqnarray*}
\mmp\{\underset{1\leq j\leq N(Tt)-1}{\sup}\,\tau_j >\varepsilon
t^c\}&\leq& \mmp\{\underset{1\leq j\leq
[Tt(\mu^{-1}+\delta)]}{\sup}\,\tau_j
>\varepsilon t^c\} +\mmp\{N(Tt)>Tt(\mu^{-1}+\delta)\} \ \to \ 0.
\end{eqnarray*}
The proof is complete.
\end{proof}

\section{Appendix}

\subsection{Probabilistic tools}

\begin{lemma}\label{diff}
For any $0\leq a<b$, any $T>0$ and any $c>0$
\begin{equation}\label{di}
{\underset{u\in [0,\,T]}{\sup}\,\big({N(ut-a)-N(ut-b)}\big)\over
t^c} \tp \ 0, \ \ t\to\infty.
\end{equation}
\end{lemma}
\begin{rem}
A perusal of the proof reveals that the rate of convergence to
zero in \eqref{di} is not optimal. However, the present form of
\eqref{di} serves our needs. In general, it seems very likely that
the actual rate of a.s. convergence in \eqref{di} is the same as
in Theorem 2 in \cite{Stei}. Note however that the cited result
assumed that $K_T\to\infty$ as $T\to\infty$ whereas we need
$K_T={\rm const}$.
\end{rem}
\begin{proof}
We start by writing
\begin{eqnarray*}
\underset{u\in
[0,\,T]}{\sup}\,\big({N(ut-a)-N(ut-b)}\big)&=&\underset{u\in
[0,\,T]}{\sup}\,\big({N(ut-a)-N(ut-b)}\big)1_{[b^{-1}t,\,
\infty)}(u)\\&+&\underset{u\in
[0,\,T]}{\sup}\,\big({N(ut-a)-N(ut-b)}\big)1_{[0,\,b^{-1}t)}(u)\\&=&
\underset{u\in [0,\,Tt-b]}{\sup}\,\big(N(u+b-a)-N(u)\big)\\&+&
\underset{u\in [0,\,T]}{\sup}\,N(ut-a)1_{[0,\,b^{-1}t)}(u)\\&\leq&
\underset{u\in [0,\,Tt-b]}{\sup}\,\big(N(u+b-a)-N(u)\big)+ N(b-a).
\end{eqnarray*}

To prove the equality
$$\underset{u\in [0,\,S_{N(Tt-b)-1}]}{\sup}\,\big(N(u+b-a)-N(u)\big)=
\underset{0\leq k\leq
N(Tt-b)-1}{\sup}\,\big(N(S_k+b-a)-N(S_k)\big)$$ just note that
obviously the right-hand side does not exceed the left-hand side,
and that while $u$ is traveling from $S_k$ to $S_{k-1}-$ the
numbers of $S_j$'s falling into the interval $(u,u+b-a]$ can only
decrease. In general, the following estimate holds true:
\begin{eqnarray*}
\underset{0\leq k\leq
N(Tt-b)-1}{\sup}\,\big(N(S_k+b-a)-N(S_k)\big)&\leq& \underset{u\in
[0,\,Tt-b]}{\sup}\,\big(N(u+b-a)-N(u)\big)\\&\leq& \underset{0\leq
k\leq N(Tt-b)}{\sup}\,\big(N(S_k+b-a)-N(S_k)\big)=:Z(t).
\end{eqnarray*}
A possible overestimate here is due to taking into account the
extra interval $(Tt-a, S_{N(Tt-b)}+b-a]$.

By the weak law of large numbers, for any $\delta>0$,
\begin{equation}\label{inter}
\lit\mmp\{N(Tt-b)>Tt(\mu^{-1}+\delta)\}=0,
\end{equation}
where we set $\mu^{-1}$ to equal zero if $\mu=\me\xi=\infty$. It
is known that $N(b-a)$ has exponential moments of all orders (see,
for instance, Theorem 2 in \cite{BelMaks}). Hence, for any
$\gamma>0$
$$\mmp\{N(b-a)>x\}=O(e^{-\gamma x}), \ \ x\to\infty.$$ Now we conclude that, for any $\varepsilon>0$,
\begin{eqnarray}\label{inter2}
&& \mmp\{\underset{0\leq k\leq
[Tt(\mu^{-1}+\delta)]}{\max}\,\big(N(S_k+b-a)-N(S_k)\big)>\varepsilon
t^c\}\nonumber\\&\leq&
\sum_{k=0}^{[Tt(\mu^{-1}+\delta)]}\mmp\{N(S_k+b-a)-N(S_k)>\varepsilon
t^c\}\nonumber\\&=&
([Tt(\mu^{-1}+\delta)]+1)\mmp\{N(b-a)-1>\varepsilon
t^c\}=O\big(t\exp(-\gamma\varepsilon t^c)\big)=o(1), \ \
t\to\infty.
\end{eqnarray}
Therefore, in view of \eqref{inter} and \eqref{inter2},
\begin{eqnarray*}
\mmp\{Z(t)>\varepsilon t^c\}&=&\mmp\{Z(t)>\varepsilon
t^c,\,N(Tt-b)>Tt(\mu^{-1}+\delta)\}\\&+& \mmp\{Z(t)>\varepsilon
t^c ,\,N(Tt-b)\leq Tt(\mu^{-1}+\delta)\}\\&\leq&
\mmp\{N(Tt-b)>Tt(\mu^{-1}+\delta)\}\\&+&\mmp\{\underset{0\leq
k\leq
[Tt(\mu^{-1}+\delta)]}{\max}\,\big(N(S_k+b-a)-N(S_k)\big)>\varepsilon
t^c\}\\&=& o(1), \ \ t\to\infty.
 \end{eqnarray*}
\end{proof}

\begin{lemma}\label{red}
Let $K_1, K_2: \mr^+\to\mr^+$ be nondecreasing functions such that
$K_1(t)\geq K_2(t)$, $t\in\mr^+$. Assume that 
$$\underset{t\to\infty}{\lim\sup}\,{K_1(t)+K_2(t)\over
\int_{[0,\,t]}\big(K_1(y)-K_2(y)\big){\rm d}y}\leq {\rm const}.$$
Then, for any $c>0$ and any $T>0$,
$${\underset{u\in [0,\,T]}{\sup}\,\int_{[0,\,ut]}\big(K_1(ut-y)-K_2(ut-y)\big){\rm d}N(y)\over t^c
\int_{[0,\,Tt]}\big(K_1(y)-K_2(y)\big){\rm d}y} \ \tp \ 0, \ \
t\to\infty.$$
\end{lemma}
\begin{proof}
We use the decomposition
$$\int_{[0,\,t]}\big(K_1(t-y)-K_2(t-y)\big){\rm
d}N(y)=\int_{[0,\,[t]]}+\int_{[[t],\,t]}=:I_1(t)+I_2(t).$$

\noindent For $I_2(t)$ we have $$I_2(t)\leq
\int_{[[t],\,t]}K_1(t-y){\rm d}N(y)\leq
K_1(t-[t])\big(N(t)-N([t])\big)\leq K_1(1)\big(N(t)-N(t-1)\big).$$
Hence, by Lemma \ref{diff}, for any $T>0$, $$t^{-c}\underset{u\in
[0,\,T]}{\sup}\, I_2(ut) \ \tp \ 0, \ \ t\to\infty.$$ It remains
to consider $I_1(t)$:
\begin{eqnarray*}
I_1(t)&=&
K_1(t)-K_2(t)+\sum_{j=0}^{[t]-1}\int_{(j,\,j+1]}\big(K_1(t-y)-K_2(t-y)\big){\rm
d}N(y)\\&\leq&
K_1(t)-K_2(t)+\sum_{j=0}^{[t]-1}\big(K_1(t-j)-K_2(t-j-1)\big)\big(N(j+1)-N(j)\big)\\&\leq&
K_1(t)+ \underset{s\in
[0,\,[t]]}{\sup}\,\big(N(s+1)-N(s)\big)\sum_{j=0}^{[t]-1}\big(K_1(t-j)-K_2(t-j-1)\big)\\&\leq&
K_1(t)+\underset{s\in [0,\,[t]]}{\sup}\,\big(N(s+1)-N(s)\big)
\sum_{j=0}^{[t]-1}\big(K_1([t]+1-j)-K_2([t]-1-j)\big)\\&=&
\underset{s\in
[0,\,[t]]}{\sup}\,\big(N(s+1)-N(s)\big)\bigg(\int_{[2,\,[t]]}\big(K_1(y)-K_2(y)\big){\rm
d}y+O\big(K_1(t)+K_2(t)\big)\bigg).
\end{eqnarray*}
Hence, for any $T>0$, $$\underset{u\in [0,\,T]}{\sup}\,I_1(ut)\leq
\underset{u\in [0,\,
T]}{\sup}\,\big(N(ut+1)-N(ut)\big)\bigg(\int_{[2,\,[Tt]]}\big(K_1(y)-K_2(y)\big){\rm
d}y+O\big(K_1(Tt)+K_2(Tt)\big)\bigg),$$ and, by Lemma \ref{diff},
$${\underset{u\in [0,\,T]}{\sup}\,I_1(ut)\over t^c
\int_{[0,\,Tt]}\big(K_1(y)-K_2(y)\big){\rm d}y} \ \tp \ 0, \ \
t\to\infty.$$ The proof is complete.
\end{proof}

\subsection{Analytic tools}

\begin{lemma}\label{red1}
Let $f:\mr^+\to\mr^+$ be a nondecreasing function which varies
regularly at $\infty$ with index $\gamma\geq 0$, and $f(0)=0$. Let
$\theta$ be a random variable with finite power moments of all
positive orders whose absolutely continuous law is concentrated on
$\mr^+$. Then $f^\ast: \mr^+\to\mr^+$ defined by $f^\ast(t):=\me
f((t-\theta)^+)$ is a continuous function with $f^\ast(0)=0$ and
such that $f^\ast(t)\sim f(t)$, $t\to\infty$. In particular,
$f^\ast$ varies regularly at $\infty$ with index $\gamma$.
Furthermore,
$$\int_{[0,\,t]}\big(f(y)-f^\ast(y)\big){\rm d}y \ \sim \ \me \theta f(t), \ \ t\to\infty.$$
\end{lemma}
\begin{proof}
The fact $f^\ast(0)=0$ is trivial. The continuity (even
differentiability) of $f^\ast$ follows from the representation
$$f^\ast(t)=f(0)e^{-t}+e^{-t}\int_{[0,\,t]}f(y)e^y{\rm d}y.$$
By dominated convergence, $\lit f^\ast(t)/f(t)=1$. This entails
the regular variation of $f^\ast$. Further
\begin{equation*}
\int_{[0,\,t]}\big(f(y)-f^\ast(y)\big){\rm d}y=\me
\int_{[(t-\theta)^+,\,t]}f(y){\rm d}y=\int_{[0,\,t]}f(y){\rm
d}y\mmp\{\theta>t\}+\me1_{\{ \theta\leq
t\}}\int_{[t-\theta,\,t]}f(y){\rm d}y.
\end{equation*}
As $t\to\infty$, the first term on the right-hand side tends to
$0$, by Markov inequality. The second term can be estimated as
follows $${\me f(t-\theta)\theta1_{\{\theta\leq t\}} \over f(t)}
\leq {\me1_{\{\theta\leq t\}}\int_{[t-\theta,\,t]}f(y){\rm
d}y\over f(t)}\leq \me\theta.$$ Since, as $t\to\infty$, the term
on the left-hand side converges to $\me\theta$, by dominated
convergence, the proof is complete.
\end{proof}
\begin{lemma}\label{impo}
Let $0\leq a<b<\infty$. Assume that $\lin x_n=x$ in $D$ in the
$J_1$ or $M_1$ topology. Assume also that, as $n\to\infty$, finite
measures $\nu_n$ converge weakly on $[a,b]$ to a finite measure
$\nu$, and that the limiting measure $\nu$ is continuous
(nonatomic). Then
$$\lit \int_{[a,\,b]}x_n(y)\nu_n({\rm d}y)=\int_{[a,\,b]}x(y)\nu({\rm d}y).$$
If $x$ is continuous at point $c\in [a,b]$, and $\nu=\delta_c$ is
the Dirac measure at point $c$ then $$\lin
\int_{[a,\,b]}x_n(y)\nu_n({\rm d}y)=x(c).$$
\end{lemma}
\begin{proof}
Since the convergence in the $J_1$ topology entails the
convergence in the $M_1$ topology, it suffices to investigate the
case when $\lin x_n=x$ in the $M_1$ topology.

\noindent Since $x\in D[a,b]$ the set $D_x$ of its discontinuities
is at most countable. By Lemma 12.5.1 in \cite{Whitt2},
convergence in the $M_1$ topology implies local uniform
convergence at all continuity points of the limit. Hence $E:=\{y:
\text{there exists} \ y_n \ \text{such that} \ \lin y_n= y,
\text{but} \ \lin x_n(y_n)\neq x(y)$$\}\subseteq D_x$, and, if
$\nu$ is continuous, we conclude that $\nu(E)=0$. If $x$ is
continuous at $c$ and $\nu=\delta_c$ then $c\notin E$, hence
$\nu(E)=0$. Now the statement follows from Lemma 2.1 in
\cite{broz}.
\end{proof}
For $x\in D[0,\,T]$, $T>0$, define the maximum-jump functional
\begin{equation}\label{ma}
J(x):=\underset{t\in [0,\,T]}{\sup}\,|x(t)-x(t-)|.
\end{equation}
\begin{lemma}\label{impo1}
Let $\lin x_n=x$ in the $M_1$ topology in $D[0,\,T]$, and $\lin
J(x_n)=0$. For $n\in\mn$ let $f_n:\mr^+\to\mr^+$ be nondecreasing
and absolutely continuous functions with $f_n(0)=0$. Define
$$y_n(u):=\int_{[0,\,u]}\big(x_n(y)-x(y)\big){\rm d}\big(-f_n(u-y)\big), \ \  y(u):=0, \ \ u\in [0,\,T].$$ Then $\lin y_n=y$ in the $M_1$
topology in $D[0,\,T]$.
\end{lemma}
\begin{proof}
For $z\in D[0,\,T]$ denote by $\Pi(z)$ the set of all parametric
representations of $z$ (see p.~80-82 in \cite{Whitt2} for the
definition). Since $\lin x_n=x$, Theorem 12.5.1 (i) in
\cite{Whitt2} implies that we can choose parametric
representations $(u,r)\in \Pi(x)$ and $(u_n,r_n)\in \Pi(x_n)$,
$n\in\mn$, such that
$$\lin\underset{t\in [0,\,1]}{\sup}\,|u_n(t)-u(t)|=0 \ \ \text{and} \ \
\lin\underset{t\in[0,\,1]}{\sup}\,|r_n(t)-r(t)|=0.$$ Furthermore,
according to the proof of Lemma 4.3 in \cite{WhittPang}, we can
assume that $r(t)$ is absolutely continuous with respect to the
Lebesgue measure and that
\begin{equation}\label{17}
x(r(t))r^\prime(t)=u(t)r^\prime(t) \ \ \text{a.e. on} \ \ [0,1],
\end{equation}
where $r^\prime$ is the derivative of $r$.

Clearly, $(y,r)\in \Pi(y)$. By Lemma \ref{co}(b), the functions
$y_n$, $n\in\mn$, are continuous. Hence, $(y_n(r), r)\in
\Pi(y_n)$, $n\in\mn$, and it suffices to prove that
$$\lin \underset{t\in [0,\,1]}{\sup}\,|y_n(r(t))|=0.$$ We have
\begin{eqnarray*}
y_n(r(t))&=&\int_{[0,\,r(t)]}\big(x_n(y)-x(y)\big){\rm
d}\big(-f_n(r(t)-y)\big)\\&=&
\int_{[0,\,t]}\big(x_n(r(y))-x(r(y))\big){\rm
d}\big(-f_n(r(t)-r(y))\big)\\&=&\int_{[0,\,t]}\big(x_n(r(y))-u(y)\big){\rm
d}\big(-f_n(r(t)-r(y))\big)\\&+&
\int_{[0,\,t]}\big(u(y)-x(r(y))\big)r^\prime(y)f^\prime_n(r(t)-r(y)){\rm
d}y\\&\overset{\eqref{17}}{=}&
\int_{[0,\,t]}\big(x_n(r(y))-u(y)\big){\rm
d}\big(-f_n(r(t)-r(y))\big).
\end{eqnarray*}
From the proof of Lemma 4.2 in \cite{WhittPang} it follows that
$$\lin \underset{t\in [0,\,1]}{\sup}\,|x_n(r(t))-u(t)|=0,$$ whenever $\lin x_n=x$ and $\lin
J(x_n)=0$. Hence, as $n\to\infty$, $$\underset{t\in
[0,\,1]}{\sup}\,|y_n(r(t))|\leq \underset{t\in
[0,\,1]}{\sup}\,\underset{y\in
[0,\,t]}{\sup}\,|x_n(r(y))-u(y)|f_n(r(t))=\underset{t\in
[0,\,1]}{\sup}\,|x_n(r(t))-u(t)|f_n(T) \ \to \ 0.$$
\end{proof}
\begin{lemma}\label{integr}
Let $F$ and $G$ be left- and right-continuous functions of locally
bounded variation, respectively. Then, for any real $a<b$,
$$\int_{(a,\,b]}F(y){\rm d}G(y)=F(b+)G(b)-F(a+)G(a)-\int_{(a,\,b]}G(y){\rm d}F(y)$$
\end{lemma}
\begin{proof}
This follows along the lines of the proof of Theorem 11 on p.~222
in \cite{Shir} which treats right-continuous functions $F$ and
$G$.
\end{proof}

\begin{lemma}\label{co}
(a) Let $f:\mr^+\to\mr^+$ be a continuous and monotone function
and $g:\mr^+\to\mr$ be any locally bounded function such that the
convolution $f\star g(x):=\int_{[0,\,x]}f(x-y)g(y){\rm d}y$ is
well-defined and finite. Then $f\star g$ is continuous on $\mr^+$.

\noindent (b) Let $g:\mr^+\to\mr^+$ be a continuous and
nondecreasing function and $f:\mr^+\to\mr$ be any locally bounded
function. Then the Riemann-Stieltjes convolution $f\star
g(x):=\int_{[0,\,x]}f(x-y){\rm d}g(y)$ is continuous on $\mr^+$.

\begin{proof}
(a) With $\varepsilon>0$ write for any $x\geq 0$
\begin{eqnarray*}
\big|f\star g(x+\varepsilon)-f\star g(x)\big|&\leq
&\int_{[0,\,x]}\big(f(x+\varepsilon-y)-f(x-y)\big)\big|g(y)\big|{\rm
d}y\\&+&\int_{[x,\,x+\varepsilon]}f(x+\varepsilon-y)\big|g(y)\big|{\rm
d}y\\&&
\end{eqnarray*}
As $\varepsilon\to 0$ the first integral goes to zero by monotone
convergence. The function $f$ must be integrable in the
neighborhood of zero. With this at hand it remains to note that
the second integral does not exceed $$\underset{y\in
[x,\,x+\varepsilon]}{\sup}\,|g(y)|\int_{[0,\,\varepsilon]}f(y){\rm
d}y \ \to \ |g(x+)|\times 0=0, \ \ \varepsilon\to 0.$$ The case
$\varepsilon<0$ can be treated similarly.

\noindent (b) With $\varepsilon\in (0,1)$ write for any $x\geq 0$
\begin{eqnarray*}
\big|f\star g(x+\varepsilon)-f\star
g(x)\big|&=&\int_{[0,\,x]}f(y){\rm
d}\big(-g(x+\varepsilon-y)+g(x-y)\big)\\&+&\int_{[x,\,x+\varepsilon]}f(y){\rm
d}\big(-g(x+\varepsilon-y)\big)
\end{eqnarray*}
The total variations of the integrators of the first integral are
uniformly bounded. Furthermore, in view of the continuity of $g$,
as $\varepsilon\to 0$, these integrators converge (pointwise) to
zero. Hence, as $\varepsilon\to 0$ the first integral goes to zero
by Helly's theorem for Lebesgue-Stieltjes integrals. The second
integral does not exceed
$$\underset{y\in [x,\,x+\varepsilon]}{\sup}\,|f(y)|\big(g(\varepsilon)-g(0)\big)
 \ \to \ |f(x+)|\times 0=0, \ \ \varepsilon \to
0.$$ The case $\varepsilon\in (-1,0)$ can be treated similarly.

\end{proof}

\end{lemma}

\vskip0.1cm \noindent {\bf Acknowledgement} The author is indebted
to two anonymous referees for pointing out an oversight in the
original version and other useful comments. The author thanks
Mindaugas Bloznelis for a helpful discussion and Josef Steinebach
for his comment on paper \cite{Stei}.

\end{document}